\documentclass[a4paper,12pt]{amsart}
\usepackage{diagbox}
\usepackage{graphicx,tabularx} 
\usepackage[left=2cm,top=2.5cm,right=2cm,bottom=2.5cm]{geometry}
\usepackage{amsmath}
\usepackage{amssymb}
\usepackage{amsfonts}
\usepackage{amsthm}
\usepackage{stmaryrd}
\usepackage{url}
\usepackage{hyperref}
\usepackage{fancyhdr}
\usepackage{mathrsfs}
\usepackage{mathtools}
\usepackage{color,colortbl}
\usepackage{enumerate,mdwlist}
\usepackage{tikz-cd} 

\usepackage{makecell}  

\usepackage{amstext}
\usepackage{array}

\usepackage[shortlabels]{enumitem}
\setlist[enumerate]{leftmargin=25pt}
\setlist[itemize]{leftmargin=25pt}

\newtheorem{thm}{Theorem}[section]
\newtheorem{lemma}[thm]{Lemma}
\newtheorem{prop}[thm]{Proposition}
\newtheorem{cor}[thm]{Corollary}
\newtheorem{conj}[thm]{Conjecture}

\newtheorem{ques}[thm]{Question}

\theoremstyle{definition}

\newtheorem*{rem}{Remark}
\newtheorem*{notation}{Notation}

\theoremstyle{definition}
\newtheorem{defn}[thm]{Definition}

\DeclareMathOperator{\ord}{ord}

\newcommand{\Tr}{\ensuremath{\operatorname{Tr}}}
\newcommand{\Norm}{\ensuremath{\operatorname{N}}}

\newcommand{\bra}[1]{{\left({#1}\right)}}

\newcommand{\set}[1]{{\left\{{#1}\right\}}}

\DeclarePairedDelimiter\abs{\lvert}{\rvert}

\newcommand{\al}{\ensuremath{\alpha}}

\newcommand{\ga}{\ensuremath{\gamma}}
\newcommand{\Ga}{\ensuremath{\Gamma}}

\newcommand{\ze}{\ensuremath{\zeta}}

\newcommand{\sig}{\ensuremath{\sigma}}

\newcommand{\vphi}{\ensuremath{\varphi}}
\newcommand{\om}{\ensuremath{\omega}}

\newcommand{\mc}[1]{\ensuremath{\mathcal{#1}}}

\newcommand{\ZZ}{\ensuremath{\mathbb{Z}}}

\newcommand{\FF}{\ensuremath{\mathbb{F}}}
\newcommand{\QQ}{\ensuremath{\mathbb{Q}}}

\newcommand{\NN}{\ensuremath{\mathbb{N}}}

\newcommand{\CC}{\ensuremath{\mathbb{C}}}

\newcommand{\Gal}{\ensuremath{\textnormal{Gal}}}

\newcommand{\ssq}{\ensuremath{\subseteq}}

\newcommand{\mf}{\ensuremath{\mathfrak}}
\newcommand{\ol}{\ensuremath{\overline}}

\usepackage[normalem]{ulem}

\newcommand{\bs}[1]{\ensuremath{\mathbf{#1}}}

\newcommand{\out}[1]{}

\numberwithin{equation}{section}

\title[Principal minors of Fourier matrices of square-free order]{Principal minors of Fourier matrices\\ of square-free order}

\author[A. Caragea]{Andrei Caragea}
\address{Mathematical Institute for Machine Learning and Data Science\\
		 Katholische Universit\"at Eichstätt-Ingolstadt\\
         Germany}
\email{andrei.caragea@gmail.com}

\author[D. G. Lee]{Dae Gwan Lee}
\address{Department of Applied Mathematics\\
		 Pukyong National University\\
         South Korea}
\email{daegwan@pknu.ac.kr}

\author[R. D. Malikiosis]{Romanos Diogenes Malikiosis}
\address{Department of Mathematics\\
		 Aristotle University of Thessaloniki\\
         Greece}
\email{rwmanos@gmail.com}

\author[G. E. Pfander]{G\"otz E. Pfander}
\address{Mathematical Institute for Machine Learning and Data Science\\
		 Katholische Universit\"at Eichstätt-Ingolstadt\\
         Germany}
\email{Pfander@ku.de}

\thanks{R. D. Malikiosis acknowledges that this project is carried out within the framework of the National Recovery and Resilience Plan Greece 2.0, funded by the European Union, NextGenerationEU (Implementation body: HFRI, Project Name: HANTADS, No. 14770).
D.~G.~Lee is supported by the National Research Foundation of Korea (NRF) grant funded by the Korean government (MSIT) (RS-2023-00275360). A. Caragea and G. Pfander are supported by the German Research Foundation (DFG) grants PF 450/11-1 and CA 3683/1-1. The authors would like to thank Simon Weinzierl for helpful comments and corrections. 
All authors contributed equally to this work.}


\date{\today}

\begin{document}
\begin{abstract}
    Chebotar\"ev's theorem on roots of unity states that all minors of a Fourier matrix are non-zero if and only if the order of the matrix is prime.  We establish cases in which all principal minors of Fourier matrices of square-free order are non-zero.  In a subsequent paper we discuss the case of composites containing squares.  
\end{abstract}
\maketitle

\section{Introduction and main results}\bigskip\bigskip

A fundamental result on the order $N$ \emph{Fourier  matrix} $\mc{F}_N=(\om^{ij})_{0\leq i,j\leq N-1}$ with $\om=e^{-2\pi i/N}$ is attributed to Chebotarev  \cite{Chebotarev} and states
\begin{thm}\label{thm:cheb}
    All minors of $\mc{F}_N$ are non-zero exactly if $N$ is prime.
\end{thm}



In this paper, we address the following conjecture \cite{CMN24,CL22,CL24}. 

\begin{conj}\label{conj:squarefree}
     All principal minors of $\mc{F}_N$ are non-zero exactly if $N$ is square-free.
\end{conj}

Recall that a number $N\in\NN$ is \emph{square-free} if it is divisible by no square number other than $1$. 
Necessity of $N$ being square-free follows from considering the principal minor with row and column indices $0$ and $pm$ if $N=p^2m$ for $p>1$ and $m\geq 1$.


Conjecture~\ref{conj:squarefree} was formulated independently by Caragea and Lee, based on observations from the construction of hierarchical exponential Riesz bases \cite{CL22,CL24}, and by Cabrelli et al.~\cite{CMN24}, who studied necessary and sufficient conditions for bases in finite-dimensional Hilbert spaces to be woven, arriving at the conjecture through its application to Fourier bases.
Since then, at least two other groups have independently made progress towards confirming this conjecture, see Section~\ref{sec:literature} for details.

Our first result in support of Conjecture \ref{conj:squarefree} relates to $N$ being a small multiple of a prime. 

\begin{thm} \label{thm:smallmultp}
If $N$ is square-free and of the form $2p,3p,5p,6p$ or $7p$ for  $p$ prime, then all principal minors of $\mc{F}_N$ are non-zero. 
\end{thm}

Theorem \ref{thm:smallmultp} follows from Corollary \ref{cor:smallmultp} and Proposition \ref{prop:N-6p} below. Note that the smallest square free $N$ for which Conjecture~\ref{conj:squarefree} is not resolved positively is, for two factors, $11 \times 13=143$, for three factors, $2 \times 5 \times 7=70$ and, for four factors, $2 \times 3 \times 5 \times 7=210$, all of which cannot be addressed numerically as $\mathcal{F}_N$ has $2^N$ principal minors. (Arguments relating to symmetries can be applied to reduce the number of minors that have to be checked, allowing for numerical exploration only up to approximately $N=45$.)

Central to our analysis is the following tool that is established below.

\begin{thm}\label{thm:inductionstep}
Let $N=pN'$ be a square-free number with $p \in \NN$ prime and $N' \in \NN$ such that all principal minors of $\mc{F}_{N'}$ are non-zero in characteristic $p$. Then, any $N'$-principal minor (and thus any principal minor) of $\mc{F}_N$ is non-zero.
\end{thm}



Note that principal minors are a special case of $N'$-principal minors, a notion described in Definition \ref{defn:minors} below. A minor of the complex Fourier matrix being non-zero in characteristic $p$ is described in Section~\ref{subsec:basic-facts-cyclotomic-fields}.

Theorem~\ref{thm:inductionstep} provides an inductive process to construct an infinite family of $N$ for which Conjecture \ref{conj:squarefree} holds. 
First, consider any prime $p_1$. Since the minors of $\mc{F}_{p_1}$ are non-zero by Chebotar\"ev's theorem, they would also be non-zero when considered in finite characteristic $p_2$, for all but finitely many primes $p_2$. For all $p_2$ not in this finite exception set, all principal minors of $\mc{F}_{p_1p_2}$ are non-zero. Now, assuming that primes $p_1,p_2,\dotsc,p_{k-1}$ ($k\geq2$), are chosen so that all principal minors of $\mc{F}_{p_1p_2\dotsm p_{k-1}}$ are non-zero, we observe that the principal minors of $\mc{F}_{p_1p_2\dotsm p_{k-1}}$ are also non-zero in characteristic $p_k$, for all but finitely many primes $p_k$. Theorem \ref{thm:inductionstep} implies that for these primes $p_k$, all principal minors of $\mc{F}_{p_1p_2\dotsm p_k}$ are non-zero.

At each step, we may conclude that all principal minors of $\mc{F}_{p_1p_2\dotsm p_j}$ are non-zero, for all sufficiently large $p_j$, a condition that depends on $p_1,\dotsc,p_{j-1}$. 
With $\vphi$ denoting Euler's totient function, the following theorem provides some explicit bounds for $p_j$.

\begin{thm}\label{thm:mainresult2}
      Let $N=p_1p_2\dotsm p_k$ for some primes $p_1<p_2<\dotsb<p_k$. 
      Assume that  
      \[p_{j+1}>\bra{\frac{P_{j}}{2}}^{P_{j}\vphi(P_{j})/4} 
      \quad \text{for} \;\; j = 1 , \ldots , k-1, \]
        where $P_j := p_1 p_2\dotsm p_j$.
    Then, all principal minors of $\mc{F}_N$ are non-zero.
\end{thm}






As mentioned above, if $N$ is divisible by $p^2$ for some prime $p$, then $\mc{F}_N$  has zero principal minors, in fact, for each $2 \leq r \leq N-2$ there exists an $r\times r$ principal submatrix of $\mc F_N$ whose determinant is zero (see \cite{CL24}). 

Conjecture \ref{conj:squarefree} may thus be considered as a special case of the following more general question posed in \cite{CL24}.

\begin{ques}\label{ques:generalpermutation}
For which $N\in\NN$ does there exist a permutation matrix $P$ such that all principal minors of the column permuted Fourier matrix $\mc{F}_N \, P$ are non-zero?
\end{ques}

The above discussion indicates that when $N$ is divisible by the square of a prime, one must look for non-trivial permutations. 
Subsequent to \cite{CL24}, numerical techniques have shown that Question \ref{ques:generalpermutation} can be answered in the negative for $N=16$ and then one can show that no suitable permutations exist for all $N$ that are multiples of $16$.  This leads to a refinement of Question~\ref{ques:generalpermutation}.

\begin{ques}\label{ques:generalpermutation-refined}
Does there exist a permutation matrix $P$, so that all principal minors of  $\mc{F}_N \, P$ are non-zero, exactly if $N$ is not divisible by a fourth power of a prime?
\end{ques}

These results, together with a more general analysis of the non-square-free case and of candidate permutations is the main topic of the manuscript in preparation \cite{CLMP25-2}.





In this paper, we will try to extend our knowledge on the minors of $\mc{F}_N$, focusing on the case where $N$ is square-free. Unavoidably, some algebraic and number-theoretic phenomena appear, so we will need the relevant background.

\section{Related literature}\label{sec:literature}

In the last 20 years, matrices with no zero minors have received significant attention in the field of sparse signal recovery and compressed sensing, see \cite{FR13} and references therein. Indeed, if a matrix $A\in \mathbb C^{m \times N}$ has no zero minors, then any $k$ entries of $Ax$ can be used to recover $x$ if $x$ has only $k$ non-zero entries at known locations or if $x$ has at most $k/2$ non-zero entries at unknown locations. 
Tao \cite{TaoUncertainty05} formulated this result specifically for Fourier matrices in the context of uncertainty principles.

Very recently, Loukaki focused on the case where $N$ is the product of two primes \cite{Lo24}. Building on the proof of Frenkel \cite{Fr03} for Chebotar\"ev's theorem 
and using a result of Zhang \cite{pChebotarev} on Fourier minors of prime order in fields of finite characteristic, she proved Conjecture \ref{conj:squarefree} for $N=pq$, where $p, q$ are primes with some restrictive conditions (namely that $p$ is primitive modulo $q$ and larger than the Zhang bound $\Gamma_q$).  

Emmrich and Kunis, in a very recent preprint \cite{EK25}, expand on the idea of Zhang in the case of $N=pq$ for distinct primes $p,q$ and achieve much better bounds as well as lifting the restriction of primitivity. This leads to another partial solution for Conjecture \ref{conj:squarefree} for the case $N=pq$. 

Our work is also based on a strengthening of Zhang's result and an adaptation of the proof of Chebotar\"ev's theorem provided by Tao in \cite{TaoUncertainty05} (see Section \ref{sec:pCheb} for details). We provide the first (to our knowledge) partial solution to Conjecture \ref{conj:squarefree} that goes beyond the scope of $N$ being a product of two distinct primes. We also share the insight of Loukaki that these results naturally extend to a more general class of minors of Fourier matrices (see Definition \ref{defn:minors} below).

\section{Background}

In this section, we will review some well-known results about the Fourier matrix and its minors, and also develop some definitions and propositions that will be useful in proving our main results.

We begin this Section with a helpful lemma.

\begin{lemma}\label{lem:blockdeterminant}
Let $B$ be a square matrix, having the following block structure:
\[
B = 
\bra{
\begin{array}{c|c|c|c}
a_{11}B_{11} & a_{12}B_{12} & \ldots & a_{1n}B_{1n}\\ \hline
a_{21}B_{21} & a_{22}B_{22} & \ldots & a_{2n}B_{2n}\\ \hline
\vdots & \vdots & \ddots & \vdots\\ \hline
a_{n1}B_{n1} & a_{n2}B_{n2} & \ldots & a_{nn}B_{nn}
\end{array}
},
\]
where $B_{ij}$ is an $M_i\times M_j$ matrix and $a_{ij}\in\CC$. It also holds $M_1\geq M_2\geq\dotsb\geq M_n$, and $B_{ij}$ is the submatrix of $B_{i1}$, formed by the first $M_j$ columns of $B_{i1}$. Then,
\begin{equation}\label{eq:blockdeterminant}
\det B=\prod_{k=1}^n(\det A_k)^{M_k-M_{k+1}}\prod_{i=1}^n\det B_{ii},
\end{equation}
where $M_{n+1}=0$ and $A_k=\bra{a_{ij}}_{1\leq i,j\leq k}$.
\end{lemma}

\begin{proof}
We prove this by induction on $n$. 
If $n=1$, then \eqref{eq:blockdeterminant} reduces to $\det(a_{11} B_{11})=a^{M_1}\det B_{11}$ for $a_{11} \in \CC$ and $B_{11} \in\CC^{M_1\times M_1}$, which is obviously true. 
Now, as the induction hypothesis, we assume that \eqref{eq:blockdeterminant} holds for $n = N-1$, in fact, for any $(N-1)\times (N-1)$ block arrangement with the prescribed properties, and then proceed to show that \eqref{eq:blockdeterminant} also holds for $n = N$. 
If $a_{11} = a_{12} = \ldots = a_{1N} = 0$, then $\det B = 0$ and $\det A_1 = a_{11} = 0$, so \eqref{eq:blockdeterminant} holds trivially. 
Otherwise, at least one of $a_{11}, a_{12}, \ldots, a_{1N}$ is non-zero, and we may assume without loss of generality that $a_{11} \neq 0$. 
Then, we perform the following column operations to $B$ which leave the determinant invariant: denoting by $B_j$ the $j$th column of blocks, i.e., 
\[
B_j=\bra{
\begin{array}{c}
a_{1j}B_{1j}\\ \hline
a_{2j}B_{2j}\\ \hline
\vdots\\ \hline
a_{Nj}B_{Nj}
\end{array}
},\]
we replace $B_j$ by $B_j-\frac{a_{12}}{a_{11}}B_1(M_j)$ for every $j$, where $B_1(M_j)$ is the submatrix of $B_1$ formed by the first $M_j$ columns. In this way, we obtain
\begin{align*}
\det B &=\det\bra{
\begin{array}{c|c|c|c}
a_{11}B_{11} & 0 & \ldots & 0\\ \hline
a_{21}B_{21} & \bra{a_{22}-\frac{a_{12}a_{21}}{a_{11}}}B_{22} & \ldots & \bra{a_{2n}-\frac{a_{1N}a_{21}}{a_{11}}}B_{2N}\\ \hline
\vdots & \vdots & \ddots & \vdots\\ \hline
a_{N1}B_{N1} & \bra{a_{N2}-\frac{a_{12}a_{N1}}{a_{11}}}B_{N2} & \ldots & \bra{a_{NN}-\frac{a_{1N}a_{N1}}{a_{11}}}B_{NN}
\end{array}
}\\
&=a_{11}^{M_1} \det B_{11}\det\bra{
\begin{array}{c|c|c}
\bra{a_{22}-\frac{a_{12}a_{21}}{a_{11}}}B_{22} & \ldots & \bra{a_{2N}-\frac{a_{1N}a_{21}}{a_{11}}}B_{2N}\\ \hline
\vdots & \ddots & \vdots\\ \hline
\bra{a_{N2}-\frac{a_{12}a_{N1}}{a_{11}}}B_{N2} & \ldots & \bra{a_{NN}-\frac{a_{1N}a_{N1}}{a_{11}}}B_{NN}
\end{array}
}.
\end{align*}
By the induction hypothesis, we have 
\begin{align*}
\det B 
&= a_{11}^{M_1}\det B_{11} \, \prod_{k=2}^N(\det A'_k)^{M_k-M_{k+1}}\prod_{i=2}^N\det B_{ii} \\
&= a_{11}^{M_1-M_2} \, \prod_{k=2}^N \big( a_{11} \det A'_k \big)^{M_k-M_{k+1}} \prod_{i=1}^N\det B_{ii}\\
&= (\det A_1)^{M_1-M_2} \, \prod_{k=2}^N(\det A_k)^{M_k-M_{k+1}}\prod_{i=1}^N\det B_{ii}\\
&=\prod_{k=1}^N(\det A_k)^{M_k-M_{k+1}}\prod_{i=1}^N\det B_{ii},
\end{align*}
where $A'_k$ is the $(k-1) \times (k-1)$ matrix given by
\[
A'_k
=
\bra{
\begin{array}{ccc}
a_{22}-\frac{a_{12}a_{21}}{a_{11}} & \ldots & a_{2k}-\frac{a_{1k}a_{21}}{a_{11}}\\ 
\vdots & \ddots & \vdots\\ 
a_{k2}-\frac{a_{12}a_{k1}}{a_{11}} & \ldots & a_{kk}-\frac{a_{1k}a_{k1}}{a_{11}}
\end{array}
}.
\]
This completes the proof. 
\end{proof}

\subsection{Fourier minors}

Chebotar\"ev proved Theorem~\ref{thm:cheb} in 1926. Since then, many alternative proves were given over the last century. It is certainly not true for $\mc{F}_N$ with $N\in\NN$ composite as discussed above, 
however, we will try to present a family, as encompassing as possible, of non-zero minors
of $\mc{F}_N$, under the number theoretic assumption that $N$ is square-free.

\begin{defn}\label{defn:minors}
Let $A\in\CC^{N\times N}$ with $N\in\NN$, and let $K,L\subset\ZZ_N$ be sets of equal cardinality, i.e., $\abs{K}=\abs{L}$. 
We denote by $B=A[K, L]\in\CC^{M\times M}$ the square submatrix of $A$ formed with rows and columns indexed by $K$ and $L$, respectively.
For a divisor $d$ of $N$ and $0\leq i\leq d-1$, let 
\[K_{d,i}=\set{k\in K: k\equiv i\bmod d}, \;\;\; L_{d,i}=\set{\ell\in L: \ell\equiv i\bmod d}.\]
The matrix $B$ is 
\begin{itemize}
\item a \emph{principal} submatrix of $A$ if $K = L$;
\item a \emph{$d$-principal} submatrix of $A$ if $\abs{K_{d,i}}=\abs{L_{d,i}}$ for all $0\leq i\leq d-1$;
\item a \emph{$d$-Galois principal} submatrix of $A$ if there exists a Galois permutation, i.e., 
a permutation $\sig\in S_d$ of the form $\sigma(i)=si\bmod d$, $i\in\ZZ_d$, where $s\in\ZZ_d^*$ is coprime with $d$, such that $\abs{K_{d,i}}=\abs{L_{d,\sig(i)}}$ for all $0\leq i\leq d-1$.
\end{itemize}
\end{defn}

It is clear that every principal submatrix is $d$-principal for any divisor $d$ of $N$, and every $d$-principal submatrix is $d$-Galois principal. 

We now present two technical results needed for our investigation: one concerning the relationship between the invertibility of complementary submatrices, and the other involving Kronecker products. Both of these results are likely part of folklore, at least partially, but we include their proofs for the sake of completeness.

\begin{prop}\label{prop:unitarycomplementarity}
Let $A\in \CC^{N\times N}$ be a unitary matrix. If $K,L \subset\ZZ_N$ are sets of equal cardinality, then $A[K,L]$ is invertible if and only if $A[K^C,L^C]$ is invertible. 
\end{prop}

\begin{proof}
We make use of a version of the Jacobi identity for an arbitrary invertible square matrix $A$ (see e.g., (0.8.4.1) in \cite{HJ13}) which states that 
\[\det A[K,L] = (-1)^{\sum K +\sum L} \cdot \det A \cdot \det \big((A^{-1})[L^C, K^C]\big).\]

Since we additionally have that $A$ is unitary, $A^{-1}=\bar{A}^T$. This implies that 
\[(A^{-1})[L^C,K^C]=\left(\bar{A}^T\right)[L^C,K^C]=\left( \bar{A}\right)[K^C,L^C]=\overline{A[K^C,L^C]}.\]

Then, from the Jacobi identity, we have that 
\[\begin{split}
\det A[K,L] &= (-1)^{\sum K+\sum L}\cdot\det A\cdot \det \big((A^{-1})[L^C,K^C]\big) \\
& = (-1)^{\sum K+\sum L}\cdot\det A\cdot \det \Big(\overline{ A[K^C,L^C]}\Big) \\
& = (-1)^{\sum K+\sum L}\cdot\det A\cdot \overline{ \det \big(A[K^C,L^C]\big)} .
\end{split}\]
Note that $\det A \neq 0$ since $A$ is unitary, and therefore, we have $\det A[K,L] \neq 0$ if and only if $\det \big(A[K^C,L^C]\big ) \neq 0$, as desired. 
\end{proof}

\begin{rem}
\,

    \begin{itemize}
        \item The proof of Proposition \ref{prop:unitarycomplementarity} was suggested to us by Loukaki and can also be found in \cite{Lo24}. 
        \item Proposition \ref{prop:unitarycomplementarity} remains valid if the elements of $A$ belong in an algebraic number field (for example, a cyclotomic field), reduced modulo a prime ideal over a (rational) prime $q$. More precisely, if the reduced value of $\det A$ is non-zero in characteristic $q$ (for more detail, see Subsection \ref{subsec:basic-facts-cyclotomic-fields}), then for any sets $K,L \subset\ZZ_N$ of equal cardinality, the matrix $A[K,L]$ is invertible if and only if $A[K^C,L^C]$ is invertible. 
        In particular, this holds for $A = \mc F_p$ over the field $\ZZ_q$ with $p$ and $q$ different primes, since $\abs{\det\mc F_p}^2=p^p$ is not divisible by $q$.  
        \item The implication of Proposition \ref{prop:unitarycomplementarity} most relevant to our purposes is that, to prove that $\mc F_N$ has no zero principal (or, respectively, no zero $d$-principal or no zero $d$-Galois principal) minors, it suffices to check principal (or, respectively, $d$-principal or $d$-Galois principal) minors of size $r\leq\lfloor\tfrac{N}{2}\rfloor$. 
    \end{itemize}
\end{rem}    

We now turn to the second technical result which essentially says that if $N=nm$ is a composite number with $\gcd(n,m)=1$, then invertibility of submatrices of $\mc F_{mn}$ is equivalent to invertibility of the corresponding submatrices of $\mc F_n\otimes\mc F_m$.

Firstly, for $\omega=e^{2\pi i/N}$ (a principal root of unity of order $N$) and arbitrary permutations $\rho,\tau\in S_N$, we will denote by \[\mc F_N^{\rho,\tau}:=\big(\omega^{\rho(k)\cdot\tau(\ell)}\big)_{0\leq k,\ell\leq N-1}\] the matrix obtained by permuting the rows and columns of $\mc F_N$ by $\rho$ and $\tau$, respectively. If we only permute the rows, we will simply write $\mc F_N^{\rho} := \mc F_N^{\rho,\text{id}}$.

When $\tau=\rho$, the principal (respectively $d$-principal, $d$-Galois principal) submatrices of $\mc F_N^{\rho,\rho}$ are just the reordered principal (respectively $d$-principal, $d$-Galois principal) submatrices of $\mc F_N$. 

Next, for sets $K,L\subset\ZZ_N$ of equal cardinality, we can view $\det\mc F_N[K,L]$ as an integer coefficient polynomial evaluated at $\omega$. Then, we have $\det \mc F_N[K,L] = 0$ if and only if $\det\mc F_N[K,L]$ is divisible by the minimal polynomial of $\omega$, namely the $N$-th cyclotomic polynomial. In fact, since the roots of the $N$-th cyclotomic polynomial are precisely $\omega^s$ with $\gcd(s,N)=1$ (see Section~\ref{subsec:basic-facts-cyclotomic-fields} for more details), we have 
\[
\det \mc F_N[K,L] = 0 \quad \Leftrightarrow \quad \det \mc F_N^{\sigma,\text{id}} [K,L] = 0 \quad \Leftrightarrow\quad \det \mc F_N^{\text{id},\sigma}[K,L] = 0
\] 
whenever $\sigma \in S_N$ is a Galois permutation, i.e., a permutation of the form $\sigma(j)=sj\bmod N$, $j\in\ZZ_N$, where $s$ is coprime with $N$.

Now assume that $N=mn$ and that $\gcd(n,m)=1$. By the Chinese remainder theorem, there is a group isomorphism $\psi : \ZZ_m\times\ZZ_n \rightarrow \ZZ_{mn}$ that maps $(a,b) \in \ZZ_m\times\ZZ_n$ to the unique $j\in\ZZ_{mn}$ satisfying $j\equiv a\bmod m$ and $j\equiv b\bmod n$. 
We define the permutation 
$\rho: \ZZ_{mn} \rightarrow \ZZ_{mn}$ by $\rho (n k_1 + k_2) = \psi(k_1,k_2)$ for $k_1 \in \ZZ_m$ and $k_2 \in \ZZ_n$, i.e.,
\[
\rho:=\begin{pmatrix} 0 & 1 & \dots & n-1 & n & \dots & 2n-1 & \dots & mn-1 \\ \psi(0,0) & \psi(0,1) & \dots & \psi(0,n-1) & \psi(1,0) & \dots & \psi(1,n-1) & \dots & \psi(m-1,n-1)\end{pmatrix} . 
\]
Note that since $\gcd(m+n,mn)=1$, the map $\tau:\ZZ_{mn}\to\ZZ_{mn}$ defined by $\tau(j)\equiv (m+n)j\bmod mn$ is a Galois permutation. 
Putting the above observations together, we deduce that all principal (or, respectively, $d$-principal or $d$-Galois principal) submatrices of $\mc F_{mn}$ are non-singular if and only if all principal (or, respectively, $d$-principal or $d$-Galois principal) submatrices of $\mc F_{mn}^{\tau\circ\rho,\rho}$ are non-singular. 

Now, let $(k_1,k_2),(\ell_1,\ell_2)\in \ZZ_m\times\ZZ_n$ be arbitrary. Set $\eta=\omega^m$ and $\zeta=\omega^n$, so that $\eta^n=\zeta^m=1$. 
Then the $(n k_1 + k_2, n \ell_1 + \ell_2)$ entry in $\mc F_{mn}^{\tau\circ\rho,\rho}$ is 
\[
\begin{split}
\omega^{(\tau\circ\rho)(n k_1 + k_2)\cdot\rho(n \ell_1 + \ell_2)}
&=\big(\omega^{m+n})^{\rho(n k_1 + k_2)\rho(n \ell_1 + \ell_2)}=(\eta\zeta)^{\psi(k_1,k_2) \psi(\ell_1,\ell_2)} \\
&= \eta^{\psi(k_1,k_2) \psi(\ell_1,\ell_2)} \zeta^{\psi(k_1,k_2) \psi(\ell_1,\ell_2)}
=\eta^{k_2\ell_2}\zeta^{k_1\ell_1} , 
\end{split}
\]
which shows that 
\begin{equation}\label{eq:kronprod}
\mc F_{mn}^{\tau\circ\rho,\rho} = 
\bra{
\begin{array}{c|c|c|c}
\mc{F}_n & \mc{F}_n & \ldots & \mc{F}_n\\ \hline
\mc{F}_n & \ze\mc{F}_n & \ldots & \ze^{m-1}\mc{F}_n\\ \hline
\vdots & \vdots & \ddots & \vdots\\ \hline
\mc{F}_n & \ze^{m-1}\mc{F}_n & \ldots & \ze^{(m-1)^2}\mc{F}_n
\end{array}
} .
\end{equation}
Therefore, we have 
\begin{lemma}\label{lem:principalpermutation}
    The principal minors of $\mc{F}_m\otimes\mc{F}_n$ are Galois conjugates  of the principal minors of
    $\mc{F}_{mn}$, up to a sign. In particular, $\mc{F}_{mn}$ has a zero principal minor if and only if
    $\mc{F}_m\otimes\mc{F}_n$ has a zero principal minor.
\end{lemma}

The statement of Lemma~\ref{lem:principalpermutation} remains valid if `principal' is replaced with `$d$-principal' or `$d$-Galois principal', respectively. 

\begin{lemma}\label{lem:mnprincipalpermutation}
    Let $d$ be a divisor of $mn$.
    The $d$-principal (respectively $d$-Galois principal) minors of $\mc{F}_{mn}$ are Galois conjugates of the $d$-principal (respectively $d$-Galois principal)
    minors of $\mc{F}_m\otimes\mc{F}_n$. Moreover, all $d$-principal (respectively $d$-Galois principal) minors of $\mc{F}_{mn}$ are non-zero
    if and only if all $d$-principal (respectively $d$-Galois principal) minors of $\mc{F}_m\otimes\mc{F}_n$ are non-zero.
\end{lemma}




\subsection{Basic facts on cyclotomic fields}\bigskip\bigskip
\label{subsec:basic-facts-cyclotomic-fields}

The material included in this subsection related to cyclotomic fields may be found in any introductory textbook in algebraic number theory \cite{Marcus} or in the theory of cyclotomic fields \cite{washington}.

For any $n \in \NN$, there exists a unique irreducible monic polynomial $\Phi_n(X)\in\ZZ[X]$ which is a divisor of $x^n-1$ but not a divisor of $x^k-1$ for any $k < n$. This polynomial is called the \emph{$n$-th cyclotomic polynomial}, and is of degree $\vphi(n)$ which is the number of positive integers less than $n$ and coprime to $n$ (the function $\vphi: \NN \rightarrow \NN$ is known as Euler's totient function). 
In fact, the roots of $\Phi_n(X)$ are precisely the $n$-th primitive roots of unity, i.e., $\om^k$ for $1 \leq k < n$ with $\gcd(k,n) = 1$, where $\om :=e^{-2\pi i/n}$. 

We consider the extended ring $\ZZ[\om]$, and the extended field $\QQ(\om)$ called the \emph{cyclotomic field of order $n$}. 
The $n$-th cyclotomic polynomial $\Phi_n(X)$ can be characterized as the monic polynomial with integer coefficients that is the minimal polynomial of $\om$ over $\QQ$. 

The extension $\QQ(\om)/\QQ$ is of degree $\vphi(n)$, and is Galois with Galois group
\[\Gal(\QQ(\om)/\QQ)\cong \ZZ_n^*, \]
where the element $g\in\ZZ_n^*$ corresponds to the isomorphism $\tau_g \in \Gal(\QQ(\om)/\QQ)$ satisfying $\tau_g(\om)=\om^g$.   
The norm and trace of $\al\in\QQ(\om)$ are given by 
\[\Norm^{\QQ(\om)}_\QQ(\al)=\prod_{\sig\in\Gal(\QQ(\om)/\QQ)}\sig(\al), \qquad
\Tr^{\QQ(\om)}_\QQ(\al)=\sum_{\sig\in\Gal(\QQ(\om)/\QQ)}\sig(\al),\]
which will be simply written as $\Norm(\al)$ and $\Tr(\al)$, respectively, when the field extension is unambiguous. 
These are generally rational numbers, but if $\al$ is an algebraic integer, then $\Norm(\al), \Tr(\al)\in\ZZ$, and in particular, 
\begin{equation}\label{eq:normrootofunity}
    \Norm(\om)=\begin{cases}
        \phantom{-}1, & n\neq2\\
        -1, & n=2.
    \end{cases}
\end{equation} 
The norm and trace are invariant under Galois conjugation, that is, $\Norm(\sig(\al)) = \Norm(\al)$ and $\Tr(\sig(\al)) = \Tr(\al)$ for any $\sig\in\Gal(\QQ(\om)/\QQ)$ and $\al\in\QQ(\om)$.

The ring $\ZZ[\om]$ may not always be a unique factorization domain, but every ideal in $\ZZ[\om]$ splits uniquely into a product of prime ideals that are also maximal. For example, if $q$ is a (rational) prime, then the ideal generated by $q$ in $\ZZ[\om]$ splits as follows:
\[q\ZZ[\om]=(\mf{Q}_1\mf{Q}_2\dotsm\mf{Q}_r)^e,\]
where $\mf{Q}_1,\dotsc,\mf{Q}_r$ are distinct prime ideals in $\ZZ[\om]$, $e:=\vphi(q^a)$ is the \emph{ramification index}, and $q^a$ is the maximal power of $q$ dividing $n$, i.e., $n=q^a m$ with $q\nmid m$. In fact, we have $ref=\vphi(n)$, where $f:=\ord_m(q)$ is the multiplicative order of $q$ modulo $m$, meaning that $q^f$ is the smallest power of $q$ that is $1\bmod m$. 
For $\mf{Q}=\mf{Q}_j$ with any $1\leq j\leq r$, the residue field
$\ZZ[\om]/\mf{Q}$ is known to be a finite extension of $\FF_q$ of degree $f$, i.e., $\ZZ[\om]/\mf{Q}\cong\FF_{q^f}$. 

When considering the Fourier matrix $\mc{F}_n=(\om^{ij})_{0\leq i,j\leq n-1}$, we observe that all of its entries are $n$-th roots of unity. 
We will assume that $q$ is a prime with $q\nmid n$, which, in the above setting, corresponds to $n = q^0 m$, so that $m=n$, $a=0$, and $e=1$ (i.e., each prime ideal over $q$ is \emph{unramified} in $\QQ(\om)$), and therefore $rf=\vphi(n)$ with $f=\ord_n(q)$. 
Note that for any non-zero $\xi \in \FF_{q^f}$, we have $\xi^{q^f - 1} = 1$ in $\FF_{q^f}$, and since $q^f \equiv 1\bmod n$, i.e., $q^f - 1 = k n$ for some $k \in \ZZ$, we get $(\xi^k)^n = 1$ in $\FF_{q^f}$. 
Since $q^f$ is the smallest power of $q$ that is $1\bmod n$, we deduce that $\FF_{q^f}$ is the smallest field of characteristic $q$ containing $n$-th roots of unity. 
The Fourier matrix $\mc{F}_n$ in characteristic $q$ can be obtained by taking the entries of $\mc{F}_n$ modulo $\mf{Q}$, where $\mf{Q}=\mf{Q}_j$ with any $1\leq j\leq r$, so that the entries are elements of $\ZZ[\om]/\mf{Q}\cong\FF_{q^f}$. 

Alternatively, one can directly define the Fourier matrix $\mc{F}_n$ in characteristic $q$ (which we denote by $\ol{\mc{F}}_n$) by fixing an $n$-th root of unity in $\FF_{q^f}$, say $\ze$, and then setting $\ol{\mc{F}}_n := (\ze^{ij})_{0\leq i,j\leq n-1}$. 

Both approaches yield the Fourier matrix $\mc{F}_n$ in characteristic $q$, but we will mostly use the latter construction.

\begin{notation}
For any $n\times n$ matrix $\mc F$ with rows and columns indexed by $\ZZ_n$, we denote by $\mc{F}[K,L]$ the submatrix of $\mc{F}$ whose rows and columns are indexed by $K, L\subset\ZZ_n$, respectively.

Throughout, we will frequently use the minors of $\mc{F}_n$ and $\ol{\mc{F}}_n$.
For any sets $K, L\subset\ZZ_n$ with $\abs{K}=\abs{L}$, we define $D_{K,L}:=\det(\mc{F}_n[K,L]) = \det (\om^{ij})_{i \in K, \, j \in L} \in \CC$, and similarly, if $q$ is a prime with $q\nmid n$, we define $\ol{D}_{K,L}:=\det(\ol{\mc{F}}_n[K,L]) = \det (\ze^{ij})_{i \in K, \, j \in L} \in \FF_{q^f}$. 



We will also need the minors of $\mc{V}_n:=\bra{X^{ij}}_{0\leq i,j\leq n-1}$.
For any sets $K, L\subset\ZZ_n$ with $\abs{K}=\abs{L}$, we define the polynomial $P_{K,L}(X) := \det(\mc{V}_n[K,L]) = \det (X^{ij})_{i \in K, \, j \in L} \in\ZZ[X]$. 
For a polynomial $P(X) \in \ZZ[X]$, we define $\ol{P}(X) \in \FF_q[X]$ to be the polynomial obtained by taking the coefficients of $P(X)$ modulo $q$, so that all its coefficients belong to $\FF_q$.
\end{notation}

Noting that $D_{K,L} = P_{K,L}(\om)$, and that $\Phi_n(X)$ is the minimal polynomial of $\om$ over $\QQ$, we have 
\[
D_{K,L} = 0 
\quad \Leftrightarrow \quad 
P_{K,L}(\om)=0
\quad \Leftrightarrow \quad 
\Phi_n(X)\mid P_{K,L}(X) . 
\]

In the finite field setting, we also have $\ol{D}_{K,L} = \ol{P}_{K,L}(\ze)$, so that 
\[
\ol{D}_{K,L} = 0 
\quad \Leftrightarrow \quad 
\ol{P}_{K,L}(\ze)=0
\]
but the condition $\ol{P}_{K,L}(\ze)=0$ does not necessarily imply $\ol{\Phi}_n(X)\mid \ol{P}_{K,L}(X)$, unless $f=\vphi(n)$. This is because the irreducible polynomial of $\ze$ over $\FF_q$ is
\[(X-\ze)(X-\ze^q)(X-\ze^{q^2})\dotsm(X-\ze^{q^{f-1}}),\]
which is a polynomial of degree $f$, while $\deg \ol{\Phi}_n(X) = \vphi(n) = rf$. 
On the other hand, we have 
\[
\ol{D}_{K,L} = 0 
\quad \Leftrightarrow \quad 
D_{K,L}\in\mf{Q} , 
\]
when $\ol{\mc{F}}_n$ is interpreted as the matrix obtained by reducing the entries of $\mc{F}_n$ modulo $\mf{Q}$.

We would like to point out that these two approaches of interpreting zero minors in finite characteristic (arising from the two equivalent definitions of $\ol{\mc{F}}_n$) are essentially the same, and this is evident by the commutativity of the following diagram: 
\[ \begin{tikzcd}
\ZZ[X] \arrow{r}{\eta} \arrow[swap]{d}{\varrho} & \ZZ[\om] \arrow{d}{\ol\varrho} \\%
\FF_q[X] \arrow{r}{\ol\eta}& \FF_{q^f}
\end{tikzcd}
\]
where $\eta, \ol{\eta}$ are the evaluation maps $\eta(P)=P(\om)$ and $\ol{\eta}(\ol{P})=\ol{P}(\ze)$, $\varrho$ is simply $\varrho(P)=\ol{P}$ and finally, $\ol\varrho$ is the reduction $\bmod \mf{Q}$.

If $\ol{D}_{K, L}=0$, then $D_{K,L}\in\mf{Q}$, and the Galois group $\Gal(\QQ(\om)/\QQ)$ acts transitively on the prime ideals dividing $q$, so that $q\mid \Norm(D_{K, L})$. Conversely, if $q\mid\Norm(D_{K, L})$, then at least one of the Galois conjugates of $D_{K, L}$ belongs to $\mf{Q}$. Since Galois automorphisms are induced by $\om\mapsto\om^k$ with some $k\in \ZZ_n^*$, the Galois conjugates of $D_{K, L}$ can be expressed as $D_{K, \ell L}$ or even $D_{k K, \ell L}$, for some $k, \ell\in\ZZ_n^*$, which is again a minor of $\mc{F}_n$; thus, we have $D_{K', L'} \in\mf{Q}$ for some $K', L'\subset\ZZ_n$. herefore, if $q\mid\Norm(D_{K, L})$, then $\ol{D}_{K', L'}=0$ for some $K', L'\subset\ZZ_n$. 
Motivated by these observations, we introduce the following definition.

\begin{defn}
Let $M\in\NN$ and $q$ a prime such that $q\nmid M$. We say that $\mc{F}_M$ satisfies the \emph{$q$-Chebotarev property} if one of the following equivalent conditions hold:
\begin{enumerate}
    \item All minors of $\mc{F}_M$ are non-zero in characteristic $q$.
    \item All minors of $\ol{\mc{F}}_M$ (with an $M$-th root of unity $\ze$ chosen from $\FF_{q^f}$) are non-zero.
    \item For every $K,L\subset\ZZ_M$ with $\abs{K}=\abs{L}$, the norm of $D_{K,L}$ is not divisible by $q$, i.e., $q\nmid \Norm(D_{K,L})$.
    \item For every $K,L\subset\ZZ_M$ with $\abs{K}=\abs{L}$, we have $\ol{D}_{K, L}\neq0$.
\end{enumerate}
\end{defn}

If we translate the index set of rows or columns, the respective determinant changes by a power of $\om$ (up to a $\pm$ sign, depending on the ordering of elements in the translated set). Indeed, if $K'=K+k'$ for some $k'\in\ZZ_n$, then 
\[D_{K',L}=\pm\det(\om^{(k+k')\ell})_{k\in K, \ell\in L}=\pm\om^{k'\sum_{\ell\in L}\ell}D_{K,L}.\]
Using \eqref{eq:normrootofunity}, we obtain:

\begin{prop}\label{prop:affinenorminvariance}
    Let $n \in \NN$ with $n \geq 3$, and let $\psi$ be an affine transformation on $\ZZ_n$, i.e., $\psi : \ZZ_n \rightarrow \ZZ_n$ given by $\psi(x)=ax+b$ for some $a\in\ZZ_n^*$ and $b\in\ZZ_n$. 
    Then for any $K, L\ssq\ZZ_n$ with $\abs{K} = \abs{L}$, we have 
    \[N(D_{K,L})=N(D_{\psi(K),L})=N(D_{K,\psi(L)}).\]
\end{prop}

We conclude this Section, by listing some important values of cyclotomic polynomials:
\begin{equation}\label{eq:valuecyclotomic0}
\Phi_n(1)=\prod_{k\in\ZZ_n^*}(1-\om^k)=\Norm(1-\om)=\begin{cases}
    p, & \text{if $n$ is a power of the prime }p\\
    1, & \text{otherwise}.
\end{cases}
\end{equation}

Moreover, if $n$ is square-free and $p$ is a prime dividing $n$, then 
\begin{equation}\label{eq:valuecyclotomic}
    \Phi_n(\om^p)=pu,
\end{equation}
where $u$ is a unit in $\ZZ[\om^p]$ (see \cite{KurshanOdlyzko}).


%

\bigskip\bigskip

\section{Exploring further the \texorpdfstring{$q$}{}-Chebotar\"ev property}\label{sec:pCheb}\bigskip\bigskip

In this section, we will address the obstruction that arises in both proofs of the Conjecture \ref{conj:squarefree} for products of two primes that have appeared so far \cite{EK25,Lo24}, namely the absence of the $q$-Chebotar\"ev property. 
When $p$ is prime, all minors of $\mc{F}_p$ are non-zero by Chebotarev's Theorem, so the matrix $\mc{F}_p$ satisfies the $q$-Chebotarev property for all but finitely many primes $q\neq p$.

In \cite{pChebotarev}, an explicit lower bound for such $q$ was derived under the additional assumption that $q$ is a primitive element in $\bmod \; p$, meaning that the multiplicative order of $q$ modulo $p$ is equal to $\vphi(p) = p-1$.
We will extend this result regardless of $\ord_p(q)$; this will be accomplished by modifying the proof in \cite{pChebotarev}.

Note that $\mc{F}_p$ can be expressed in any field that contains the primitive $p$-th roots of unity.
Therefore, when considering $\mc{F}_p$ in characteristic $q$, we may take the smallest extension of $\FF_q$ that contains the $p$-th roots of unity, namely the field $\FF_{q^r}$ with $r=\ord_p(q)$, or any other field that contains $\FF_{q^r}$\footnote{This is equivalent to considering the elements of $\mc{F}_p$ modulo a prime ideal in $\QQ(e^{-2\pi i/p})$ that is above $q$.}.

The main tool in \cite{pChebotarev} is the following Lemma, which we provide here with a simpler proof.

\begin{lemma}\label{lem:localtoglobal}
Let $p$ and $q$ be distinct prime numbers. Let $f(X)\in\ZZ_{\geq0}[X]$ and $q>f(1)$. If $\bar{f}(X)$ is divisible by $\sum_{i=0}^{p-1}X^i$ in $\ZZ_q[X]$, then $f(X)$ is divisible by $\sum_{i=0}^{p-1}X^i$ in $\ZZ[X]$.
\end{lemma}

\begin{proof}
Consider the Euclidean division of $f(X)$ by $X^p-1$ in $\ZZ[X]$:
\[
f(X) = g(X) (X^p-1) + r(X) , 
\]
where $g(X), r(X) \in \ZZ[X]$ and $\deg r(X) \leq p-1$.
Note that $f(X)\in\ZZ_{\geq0}[X]$ implies $r(X)\in\ZZ_{\geq0}[X]$. 
Indeed, we have $X^n\equiv X^s\bmod(X^p-1)$ whenever $n=kp+s$ with $k \in \ZZ$ and $s \in \{ 0, \ldots, p-1 \}$, so if
\[f(X)=\sum_{i=0}^{\deg f}f_iX^i, \;\;\; r(X)=\sum_{i=0}^{p-1}r_iX^i,\]
then  
\[r_i=\sum_{j\equiv i\bmod p}f_j  \; \geq 0 ,\]
which shows that $r(X)\in\ZZ_{\geq0}[X]$. Next, we observe that
\begin{equation}\label{cyclicdivisibility}
P(X)\cdot\sum_{i=0}^{p-1}X^i\equiv\sum_{i=0}^{p-1}P(1)X^i\bmod(X^p-1),
\end{equation}
hence, $f(X)$ is divisible by $\sum_{i=0}^{p-1}X^i$ in $\ZZ[X]$, if and only if 
\[
r_0=r_1=\dotsb=r_{p-1}.\]
Equation \eqref{cyclicdivisibility} holds also in characteristic $q$, so the fact that $\sum_{i=0}^{p-1}X^i$ divides $\bar{f}(X)$ in $\ZZ_q[X]$ implies that 
\[
r_0\equiv r_1\equiv \dotsb\equiv r_{p-1}\bmod q .
\]
Denoting this common residue value by $c \in \{ 0, \ldots, q-1 \}$, we may write $r_i=s_i q + c$ for some integer $s_i\geq0$, so that 
\[f(1)=r(1)=\sum_{i=0}^{p-1}r_i = q\sum_{i=0}^{p-1}s_i + pc.\]
Since $q>f(1)$, we must have $s_i=0$ for all $i$, thus showing that $r_i=c$ for all $i$, completing the proof.
\end{proof}


We now proceed to the proof of the main theorem in this section; we define $V_{\bs a}(X_1, X_2, \dotsc, X_n)$ to be the generalized Vandermonde determinant of the $n\times n$ polynomial matrix whose $(i,j)$ entry is $X_i^{a_j}$. If $\bs s=(0,1,\dotsc,n-1)$, we have the standard Vandermonde determinant
\[V_{\bs s}(X_1,\dotsc, X_n)=\prod_{1\leq i<j\leq n}(X_j-X_i).\]
It is known \cite{GenVandermonde} that $V_{\bs s}(X_1,\dotsc, X_n)$ divides $V_{\bs a}(X_1,\dotsc, X_n)$ in $\ZZ[X_1,\dotsc, X_n]$; denote their quotient as $P_{\bs a}(X_1,\dotsc, X_n)$. The sum of the coefficients of $P_{\bs a}$ equals \cite{GenVandermonde}
\[\frac{V_{\bs s}(a_1,a_2,\dotsc, a_n)}{V_{\bs s}(0,1,\dotsc,n-1)}.\]
Define
\[\ga_n=\max\set{\frac{V_{\bs s}(a_1,a_2,\dotsc, a_n)}{V_{\bs s}(0,1,\dotsc,n-1)}: 0\leq a_1<a_2<\dotsb<a_n\leq p-1}\]
and 
\begin{equation}\label{eq:Gammadefinition}
    \Ga_p=\max\set{\ga_n:2\leq n\leq p-1}.
\end{equation}

\begin{thm}\label{thm:genZhang}
Let $p, q$ be distinct primes such that $r=\ord_p(q)$ and $q>\Ga_p^\frac{p-1}{r}$. If $\om$ is a primitive $p$-th root of unity in $\FF_{q^r}$, then all minors of $\mc{F}_p$ in $\FF_{q^r}$ are non-zero.
\end{thm}

\begin{proof}
If $p\leq5$, we have seen that all minors of $\mc{F}_p$ are non-zero in any characteristic $q\neq p$, so we may assume that $p\geq7$, so that $q$ is certainly odd. Suppose we have an integer $n$ with $2\leq n\leq p-1$ and integer vectors $\bs a=(a_1, a_2,\dotsc, a_n)$, $\bs b=(b_1, b_2,\dotsc, b_n)$, both with strictly increasing coordinates between $0$ and $p-1$.

Assume now that $q>\Ga_p^\frac{p-1}{r}$. In $\ZZ[X]$, we have
\[V_{\bs a}(X^{b_1}, X^{b_2}, \dotsc, X^{b_n})=P_{\bs a}(X^{b_1}, X^{b_2}, \dotsc, X^{b_n})V_{\bs s}(X^{b_1}, X^{b_2}, \dotsc, X^{b_n}).\]
Suppose that $V_{\bs a}(\om^{b_1}, \om^{b_2}, \dotsc, \om^{b_n})=0$ in $\FF$, i.e. the minor of $\mc{F}_p$ corresponding to rows indexed by $a_1,\dotsc, a_n$ and columns indexed by $b_1,\dotsc,b_n$ is zero in $\FF$. Since $V_{\bs s}(\om^{b_1}, \om^{b_2}, \dotsc, \om^{b_n})\neq0$ in any characteristic, we must have $P_{\bs a}(\om^{b_1}, \om^{b_2}, \dotsc, \om^{b_n})=0$ in $\FF$. For simplicity, denote
\[P(X)=P_{\bs a}(X^{b_1}, X^{b_2}, \dotsc, X^{b_n}).\]
The polynomial $\bar{P}(X)$ in $\ZZ_q[X]$ has $\om, \om^q,\dotsc, \om^{q^{r-1}}$ as roots. Let $m=\frac{p-1}{r}$ and consider a set of representatives of the subgroup generated by $q$ in $\ZZ_p^*$, say $1,t_1,\dotsc,t_{m-1}$. Next, define
\[f(X)=P(X)P(X^{t_1})\dotsm P(X^{t_{m-1}}),\]
so that the polynomial $\bar{f}(X)\in\ZZ_q[X]$ accepts all powers $\om,\dotsc,\om^{p-1}$ as roots, hence $\sum_{i=0}^{p-1}X^i$ divides $\bar{f}(X)$ in $\ZZ_q[X]$. We have
\[q>\Ga_p^m\geq \ga_n^m\geq P(1)^m=f(1),\]
so applying Lemma \ref{lem:localtoglobal} we get that $\sum_{i=0}^{p-1}X^i$ divides $f(X)$ in $\ZZ[X]$. Plugging in $X=1$, we get that $p$ divides $f(1)=P(1)^m$, hence $p\mid P(1)$. This shows that $p$ divides
\[V_{\bs s}(a_1,a_2,\dotsc,a_n)=\prod_{1\leq i<j\leq n}(a_j-a_i),\]
a contradiction, completing the proof.
\end{proof}

We now indicate some simple cases where a minor is non-zero in any finite characteristic.

\begin{prop}\label{prop:arithmeticprogression}
Let $q$ be a prime with $q\nmid M$. Let $A$ be an arithmetic progression of step $d$ in $\ZZ_M$, and let $B\subset\ZZ_M$ be a set with $\abs{B}=\abs{A}$ such that $d(b-b')\not\equiv0\bmod M$ for any $b\neq b'$ in $B\subset\ZZ_M$. Then $\ol{D}_{A, B}\neq0$.
\end{prop}

\begin{proof}
Assume that $A=\set{a_1,\dotsc,a_n}$, where $a_j=a+(j-1)d$, $1\leq j\leq n$, for some $a, d\in\ZZ_M$, and let $B=\set{b_1,\dotsc,b_n}$. Then,
\[
\ol{D}_{A,B}=\pm\prod_{j=1}^n\ze^{ab_j}\cdot\ol{D}_{A-a,B}=\pm\prod_{j=1}^n\ze^{ab_j}\cdot V_{\bs s}(\ze^{-db_1},\dotsc,\ze^{-db_n})\neq0. \qedhere
\]
\end{proof}

It is not hard to see that for $p=2, 3, 5$, every subset $A$ of $\ZZ_p$ is an arithmetic progression, and there exists a corresponding set $B$ satisfying the condition in Proposition~\ref{prop:arithmeticprogression}. For instance, $A=\{ 0, 1, 4 \} = \{ 6, 10, 14 \}$ is an arithmetic progression of step $4$ in $\ZZ_5$, and $B$ can be chosen as $\{ 0, 1, 2 \}$. We thus obtain the following:


\begin{prop}\label{prop:235cheb}
    Let $p=2, 3$ or $5$. Then $\mc{F}_p$ has the $q$-Chebotarev property for all primes $q\neq p$.
\end{prop}

On the other hand, there exists a subset of $\ZZ_7$ that is not an arithmetic progression, namely $\set{0,1,3}$. However, a quick numerical search yields the following: 

\begin{prop}\label{prop:7cheb}
    The matrix $\mc{F}_7$ has the $q$-Chebotarev property for all primes $q\neq 2, 7$.
\end{prop}

\begin{proof}
    By the remark following Proposition \ref{prop:unitarycomplementarity}, it suffices to show that for every $A,B\subset\ZZ_7$ with $\abs{A}=\abs{B} \leq 3$, the norm of $D_{A,B}$ is not divisible by any prime $q\neq 2, 7$. 
    The case $\abs{A}=\abs{B} = 1$ is obvious. 
    For the case $\abs{A}=\abs{B} = 2$, note that every $2$-element subset of $\ZZ_7$ is trivially an arithmetic progression, and thus, Proposition~\ref{prop:arithmeticprogression} implies that for any prime $q \neq 7$, we have $\ol{D}_{A, B}\neq0$, equivalently, $q\nmid \Norm(D_{A,B})$. 
    
    It remains to check the case $\abs{A}=\abs{B}=3$. 
    If either $A$ or $B$ is an arithmetic progression, then similarly as above, we have $q\nmid \Norm(D_{A,B})$ for any prime $q \neq 7$. So, we assume that neither $A$ nor $B$ is an arithmetic progression in $\ZZ_7$. Since $\Norm(D_{A,B})$ is invariant under applying affine transformations on $A$ and $B$ (see Proposition~\ref{prop:affinenorminvariance}), we may further assume that both $A$ and $B$ contain $0$ and $1$. The only $3$-element subsets of $\ZZ_7$ that contain $0$ and $1$ but are not arithmetic progressions, are $\set{0,1,3}$ and $\set{0,1,5}$. However, the set $\set{0,1,3}$ is mapped to $\set{0,1,5}$ via the linear transformation $x\mapsto 5x$, so we only need to check $D_{A,A}$ with $A=\set{0,1,3}$, given by 
    \[D_{A,A}=\begin{vmatrix}
        1 & 1 & 1\\
        1 & \om & \om^3\\
        1 & \om^3 & \om^2
    \end{vmatrix}=-\om-\om^2+3\om^3-\om^6=\om^3(3-\om^3-\om^5-\om^6).\]
    Since $\set{3,5,6}$ is the set of quadratic non-residues in $\ZZ_7$, the term $3-\om^3-\om^5-\om^6$ is invariant under the maps $\om\mapsto\om^k$, where $k\in\set{1, 2, 4}$. Therefore, we have 
    \[\Norm(D_{A,A})=\Norm(\om^3)(3-\om^3-\om^5-\om^6)^3(3-\om-\om^2-\om^4)^3=2^37^3,\]
    which is indeed divisible by $2$ and $7$ but not by any other prime. This completes the proof. 
\end{proof}

We will close this Section by showing that for any prime $p\geq7$, there is some prime $q\neq p$, such that $\mc{F}_p$ does not satisfy the $q$-Chebotar\"ev property. To find such primes $q$, we must search for minors whose rows and columns are not indexed by arithmetic progressions, due to Proposition \ref{prop:arithmeticprogression}. In such cases, the minors are generalized Vandermonde determinants on $p$th roots of unity, which are, in general, hard to compute, unless the row set $\bs a$ is \emph{almost} an arithmetic progression. For example, take $\bs a=\set{0,1,2,\dotsc,n-2,n}$, so that the first $n-1$ elements are in arithmetic progression, but $\bs a$ itself is not.

    We recall \cite[equation (11)]{DVB08} that the quotient $P_{\bs a}(X_1,\dotsc,X_n)$ equals the determinant of
    \begin{equation}\label{eq:quotient}
        M=\begin{bmatrix}
    S^{a_1}(X_1) & S^{a_2}(X_1) & \ldots & S^{a_n}(X_1)\\
    S^{a_1-1}(X_1,X_2) & S^{a_2-1}(X_1,X_2) & \ldots & S^{a_n-1}(X_1,X_2)\\
    \vdots & \vdots & \ddots & \vdots \\
    S^{a_1-n+1}(X_1,X_2,\dotsc,X_n) & S^{a_2-n+1}(X_1,X_2,\dotsc,X_n) & \ldots & S^{a_n-n+1}(X_1,X_2,\dotsc,X_n)
\end{bmatrix},
    \end{equation}
    where $S^a(X_1,X_2,\dotsc,X_k)$ denotes the complete symmetric polynomial in $k$ variables of fixed homogeneous degree $a$, and equals
\[\sum_{\mathbf{p}\in(\mathbb{N}\cup\{0\})^k, |\mathbf{p}|=a}\left(\prod_{j=1}^k X_j^{p_j}\right),\]
where $\mathbf{p}=(p_1,p_2,\dotsc,p_k)$ and $|\mathbf{p}|=|p_1|+|p_2|+\dotsb+|p_k|$. It is understood that $S^a(X_1,X_2,\dotsc,X_k)=0$ if $a<0$. 

When $a_1=0, a_2=1, \dotsc, a_{n-1}=n-2$, $M$ is upper triangular and the diagonal elements are all equal to $1$ except possibly the last one; hence $\det M=S^{a_n-n+1}(X_1,\dotsc,X_n)$. If, in addition, $a_n=n$, we get $\det M=X_1+X_2+\dotsb+X_n$. This shows that if we have $A=\set{0,1,\dotsc,n-2,n}$ and $B=\set{b_1,\dotsc,b_n}$, the condition that a prime $q\neq p$ divides
\begin{equation}\label{eq:resultant}
    \Norm(\om^{b_1}+\dotsb+\om^{b_n})=R(B(X),\Phi_p(X)),
\end{equation}
implies the fact that $\mc{F}_p$ does not satisfy the $q$-Chebotar\"ev property, where $R(B(X),\Phi_p(X))$ is the resultant of $B(X)=\sum_{b\in B}X^b$ and $\Phi_p(X)$.

\begin{thm}\label{thm:Gauss}
    Let $p\geq7$ be a prime and $q$ another prime, such that $q\mid\frac{p\pm1}{4}$, where the sign is chosen such that $p\pm1$ is a multiple of $4$. Then, $\mc{F}_p$ does not satisfy the $q$-Chebotar\"ev property.
\end{thm}

\begin{proof}
    Let $n=\frac{p-1}{2}$ and $B=\set{b_1,\dotsc,b_n}$ be the set of quadratic residues $\bmod p$. If 
    \[G(p)=\sum_{k=0}^{p-1}\om^{k^2}\]
    is the standard quadratic Gauss sum $\bmod p$, we have
    \[G(p)=1+2(\om^{b_1}+\dotsb+\om^{b_n}).\]
    Using the standard formula for the Gauss sum \cite{washington}, we then have
    \[\om^{b_1}+\dotsb+\om^{b_n}=\frac{\sqrt{\pm p}-1}{2},\]
    where the sign is chosen such that $p\pm 1$ is divisible by $4$. This belongs to the unique quadratic subfield $L$ of $\QQ(\om)$, and its conjugate in $L$ is $\frac{-\sqrt{\pm p}-1}{2}$, therefore
    \[\Norm(\om^{b_1}+\dotsb+\om^{b_n})=\bra{\frac{1\pm p}{2}}^{\frac{p-1}{2}}.\]
    Since $p\geq7$, the number $\frac{p\pm1}{4}$ is greater than $1$ and coprime to $p$, 
    thus, there are primes $q\neq p$ dividing $\frac{p\pm1}{4}$, which also divide $\Norm(D_{A,B})$ with $A=\set{0,1,2,\dotsc,\frac{p-5}{2},\frac{p-1}{2}}$ and $B$ the set of quadratic residues $\bmod p$, completing the proof.
\end{proof}

\bigskip\bigskip


\section{Proof of Theorems \ref{thm:inductionstep} and \ref{thm:mainresult2} and consequences}

\bigskip\bigskip

In this Section, we will prove Theorems \ref{thm:inductionstep} and \ref{thm:mainresult2}.

\begin{proof}[Proof of Theorem \ref{thm:inductionstep}]
According to Lemma \ref{lem:mnprincipalpermutation}, we will consider the $N'$-principal minors of $\mc{F}=\mc{F}_{N'}\otimes\mc{F}_p$ instead of those of $\mc{F}_{N}=\mc{F}_{pN'}$.

Now let $\mc{M}, \mc{L}\ssq\ZZ_N$ with $\abs{\mc{M}} = \abs{\mc{L}}$, and consider the submatrix $\mc{F}[\mc{M},\mc{L}]$. If all elements of $\mc{M}$ and $\mc{L}$ have the same residue $\bmod N'$, then the minor $D_{\mc{M},\mc{L}}:=\det\mc{F}[\mc{M},\mc{L}]$ equals a minor of $\mc{F}_p$ up to a root of unity, and is hence non-zero by Chebotar\"ev's Theorem.

%


Now, suppose that the elements of $\mc{M}$ belong to at least two classes $\bmod N'$. 
Defining 
\[\mc{M}_a:=\set{b\in \mc{M}:b\equiv a\bmod N'}, 
\quad \mc{L}_a:=\set{b\in \mc{L}:b\equiv a\bmod N'}
\quad \text{for} \;\; a = 0, \ldots, N'-1 ,\]
we have 
\[M=\sum_{a=0}^{N'-1}M_a, \quad L=\sum_{a=0}^{N'-1}L_a , \]
where $M:=\abs{\mc{M}}$, $L:=\abs{\mc{L}}$, $M_a:=\abs{\mc{M}_a}$ and $L_a:=\abs{\mc{L}_a}$ for $a = 0, \ldots, N'-1$. 
Clearly, we have $L_a, M_a\leq p$ for all $a$, and the fact that $\mc{F}(\mc{M},\mc{L})$ is $N'$-principal corresponds to $M_a=L_a$ for all $a$.
Consider the polynomial $D(z_1,\dotsc,z_M)=\det\bra{z_n^m}_{m\in\mc{M}, 1\leq n\leq M}$; it holds $D_{\mc{M},\mc{L}}=D(\om^{\ell_1},\dotsc,\om^{\ell_M})$, where $\mc{L}=\set{\ell_1,\dotsc,\ell_M}$ and $\ell_j\in\mc{L}_a$ if and only if 
\[L_0+\dotsb+L_{a-1}<j\leq L_0+\dotsb+L_a.\] 
If $z_i=z_j$ for some $i\neq j$, then $D(z_1,\dotsc,z_M)=0$, therefore we have the obvious factorization
\begin{equation}\label{eq:Nfactor1}
D(z_1,\dotsc,z_M)=P(z_1,\dotsc,z_M)\prod_{1\leq i<j\leq M}(z_j-z_i),
\end{equation}
where $P(z_1,\dotsc,z_M)$ has integer coefficients.
Now, $D_{\mc{M},\mc{L}}=0$ if and only if $Q(z)=P(z^{\ell_1},\dotsc,z^{\ell_M})$ is divided by $\Phi_N(z)$, the $N$-th cyclotomic polynomial. Putting $z=\ze=\om^p$, we would get the divisibility relation in the cyclotomic ring $\ZZ[\ze]$,
\[\Phi_N(\ze)\mid P(\underbrace{1,\dotsc,1}_{M_0 \text{ terms}},\underbrace{\ze,\dotsc,\ze}_{M_1 \text{ terms}}, \underbrace{\ze^2,\dotsc,\ze^2}_{M_2 \text{ terms}}, \dotsc, \underbrace{\ze^{N'-1},\dotsc,\ze^{N'-1}}_{M_{N'-1} \text{ terms}}),\]
or equivalently, that $P(1,\dotsc,1,\ze,\dotsc,\ze,\dotsc,\ze^{q-1},\dotsc,\ze^{q-1})$ is a multiple of $pu=\Phi_N(\ze)$ in $\ZZ[\ze]$, where $u$ a unit of $\ZZ[\ze]$, by \eqref{eq:valuecyclotomic}. We will show that this is not the case.

We will distinguish the variables according to the residues $\bmod N'$, by writing 
\[z_{M_0+\dotsb+M_{a-1}+j}=z_{a,j}, \;\;\;\text{ and }\;\;\; m_{M_0+\dotsb+M_{a-1}+j}=m_{a,j}\]
if $1\leq j\leq M_a$, where
\[\mc{M}_a=\set{m_{a,1}, m_{a,2},\dotsc,m_{a,M_a}}.\]
Therefore, \eqref{eq:Nfactor1} is transformed to
\begin{equation}\label{eq:Nfactor2}
\begin{split}
& D(z_{0,1},\dotsc,z_{0,M_0},\dotsc,z_{N'-1,1},\dotsc,z_{N'-1,M_{N'-1}}) \\ 
&=P(z_{0,1},\dotsc,z_{0,M_0},\dotsc,z_{N'-1,1},\dotsc,z_{N'-1,M_{N'-1}}) \\  
& \qquad \cdot\prod_{a=0}^{N'-1}\prod_{1\leq i<j\leq M_a}(z_{a,j}-z_{a,i}) \cdot \prod_{0\leq a<b\leq N'-1}\prod_{\substack{1\leq i\leq M_a\\1\leq j\leq M_b}}(z_{b,j}-z_{a,i}). 
\end{split}
\end{equation} 
In order to evaluate the polynomial $P$ for $z_{a,i}=\ze^a$, $\forall a, i$, we will apply the partial differential operator
\begin{equation}\label{eq:Npdo}
\prod_{a=0}^{N'-1}\bra{z_{a,1}\frac{\partial}{\partial z_{a,1}}}^0\bra{z_{a,2}\frac{\partial}{\partial z_{a,2}}}\dotsm\bra{z_{a,M_a}\frac{\partial}{\partial z_{a,M_a}}}^{M_a-1}
\end{equation}
to both sides of \eqref{eq:Nfactor2}, and then evaluate at $z_{a,i}=\ze^a$ for each $a$ and $i$.

When we apply \eqref{eq:Npdo} on the right-hand side of \eqref{eq:Nfactor2}, we get a sum of polynomials, one of which is
\begin{equation}\label{eq:Nmainterm}
\begin{split}
& P(z_{0,1},\dotsc,z_{0,M_0},\dotsc,z_{q-1,1},\dotsc,z_{q-1,M_{q-1}}) \\
& \qquad \cdot \prod_{a=0}^{N'-1}\prod_{j=1}^{M_a}z_{a,j}^{j-1} \cdot \prod_{0\leq a<b\leq N'-1}\prod_{\substack{1\leq i\leq M_a\\1\leq j\leq M_b}}(z_{b,j}-z_{a,i}),
\end{split}
\end{equation}
and the rest have at least one factor of the form $z_{a,j}-z_{a,i}$, which vanish when we put all $z_{a,i}$ equal to $\ze^a$. The term \eqref{eq:Nmainterm} is obtained
\[\prod_{a=0}^{N'-1}(M_a-1)!(M_a-2)!\dotsm 1!0!\]
times, hence
\begin{equation}\label{eq:Nevaluation1}
\begin{split}
& \prod_{a=0}^{N'-1}\bra{z_{a,1}\frac{\partial}{\partial z_{a,1}}}^0\dotsm\bra{z_{a,M_a}\frac{\partial}{\partial z_{a,M_a}}}^{M_a-1} \\
& \qquad \quad D(z_{0,1},\dotsc,z_{0,M_0},\dotsc,z_{N'-1,1},\dotsc,z_{N'-1,M_{N'-1}}) \Big|_{z_{a,i}=\ze^a} \\  
& = P(\underbrace{1,\dotsc,1}_{M_0 \text{ terms}},\underbrace{\ze,\dotsc,\ze}_{M_1 \text{ terms}}, \dotsc, \underbrace{\ze^{N'-1},\dotsc,\ze^{N'-1}}_{M_{N'-1} \text{ terms}}) \\
& \qquad \quad \cdot \prod_{a=1}^{N'-1}\ze^{aM_a(M_a-1)/2} \cdot \prod_{0\leq a<b\leq N'-1}(\ze^b-\ze^a)^{M_aM_b} \cdot \prod_{a=0}^{N'-1}\prod_{k=0}^{M_a-1}k! .
\end{split}
\end{equation}
We now apply \eqref{eq:Npdo} on the left-hand side of \eqref{eq:Nfactor2}, getting
\begin{equation}\label{eq:Ndeterminant}
    \det\bra{
    \begin{array}{c|c|c|c}
    D_{0,0}(z_{0,1},\dotsc,z_{0,M_0}) & D_{0,1}(z_{1,1},\dotsc,z_{1,M_1}) & \ldots & D_{0,N'-1}(z_{N'-1,1},\dotsc,z_{N'-1,M_{N'-1}})\\ \hline
    D_{1,0}(z_{0,1},\dotsc,z_{0,M_0}) & D_{1,1}(z_{1,1},\dotsc,z_{1,M_1}) & \ldots & D_{1,N'-1}(z_{N'-1,1},\dotsc,z_{N'-1,M_{N'-1}})\\ \hline
    \vdots & \vdots & \ddots & \vdots\\ \hline
    D_{N'-1,0}(z_{0,1},\dotsc,z_{0,M_0}) & D_{N'-1,1}(z_{1,1},\dotsc,z_{1,M_1}) & \ldots & D_{N'-1,N'-1}(z_{N'-1,1},\dotsc,z_{N'-1,M_{N'-1}})
    \end{array}
    }
\end{equation}
where the block $D_{a,b}$ is a $M_a\times M_b$ matrix given by
\begin{equation}\label{eq:Nblock}
    \begin{pmatrix}
    z_{b,1}^{m_{a,1}} & m_{a,1}z_{b,2}^{m_{a,1}} & \ldots & m_{a,1}^{M_b-1}z_{b,M_b}^{m_{a,1}}\\
    z_{b,1}^{m_{a,2}} & m_{a,2}z_{b,2}^{m_{a,2}} & \ldots & m_{a,2}^{M_b-1}z_{b,M_b}^{m_{a,2}}\\
    \vdots & \vdots & \ddots & \vdots\\
    z_{b,1}^{m_{a,M_a}} & m_{a,M_a}z_{b,2}^{m_{a,M_a}} & \ldots & m_{a,M_a}^{M_b-1}z_{b,M_b}^{m_{a,M_a}}
    \end{pmatrix},
\end{equation}
or more succinctly, 
\[D_{a,b}(z_{b,1},\dotsc,z_{b,M_b})=\bra{m_{a,i}^{j-1}z_{b,j}^{m_{a,i}}}_{\substack{1\leq i\leq M_a\\1\leq j\leq M_b}}.\]

Putting $z_{a,i}=\ze^a$ for all $a$ and $i$ yields
\begin{equation}\label{eq:Nblockevaluation}
    D_{a,b}(\ze^b,\dotsc,\ze^b)=\bra{m_{a,i}^{j-1}\ze^{ab}}_{\substack{1\leq i\leq M_a\\1\leq j\leq M_b}}=\ze^{ab}\bra{m_{a,i}^{j-1}}_{\substack{1\leq i\leq M_a\\1\leq j\leq M_b}}.
\end{equation}
We consider now a permutation $\sig$ of $\ZZ_{N'}$, such that 
\[M_{\sig(0)}\geq M_{\sig(1)}\geq\dotsb\geq M_{\sig(N'-1)},\]
and permute the blocks of \eqref{eq:Ndeterminant} accordingly, we get that \eqref{eq:Ndeterminant} equals
\begin{equation}\label{eq:Ndeterminant2}
    \det\bra{
    \begin{array}{c|c|c|c}
    D_{\sig(0),\sig(0)} & D_{\sig(0),\sig(1)} & \ldots & D_{\sig(0),\sig(N'-1)}\\ \hline
    D_{\sig(1),\sig(0)} & D_{\sig(1),\sig(1)} & \ldots & D_{\sig(1),\sig(N'-1)}\\ \hline
    \vdots & \vdots & \ddots & \vdots\\ \hline
    D_{\sig(N'-1),\sig(0)} & D_{\sig(N'-1),\sig(1)} & \ldots & D_{\sig(N'-1),\sig(N'-1)}
    \end{array}
    },
\end{equation}
and when evaluated at $z_{a,i}=\ze^a$ for all $a$ and $i$, we get
\begin{equation}\label{eq:Ndeterminant3}
    \det\bra{
    \begin{array}{c|c|c|c}
    \ze^{\sig(0)^2}V_{\sig(0),\sig(0)} & \ze^{\sig(0)\sig(1)}V_{\sig(0),\sig(1)} & \ldots & \ze^{\sig(0)\sig(N'-1)}V_{\sig(0),\sig(N'-1)}\\ \hline
    \ze^{\sig(1)\sig(0)}V_{\sig(1),\sig(0)} & \ze^{\sig(1)^2}V_{\sig(1),\sig(1)} & \ldots & \ze^{\sig(1)\sig(N'-1)}V_{\sig(1),\sig(N'-1)}\\ \hline
    \vdots & \vdots & \ddots & \vdots\\ \hline
    \ze^{\sig(N'-1)\sig(0)}V_{\sig(N'-1),\sig(0)} & \ze^{\sig(N'-1)\sig(1)}V_{\sig(N'-1),\sig(1)} & \ldots & \ze^{\sig(N'-1)^2}V_{\sig(N'-1),\sig(N'-1)}
    \end{array}
    },
\end{equation}
where
\[V_{a,b}=\bra{m_{a,i}^{j-1}}_{\substack{1\leq i\leq M_a\\1\leq j\leq M_b}}.\]
By virtue of Lemma \ref{lem:blockdeterminant}, \eqref{eq:Ndeterminant3} equals
\[\prod_{k=0}^{N'-1}(\det Z_k)^{M_{\sig(k)}-M_{\sig(k+1)}}\det V_{\sig(k),\sig(k)},\]
where $M_{\sig(N')}=0$, and $Z_k$ is the square matrix formed by the first $k+1$ rows and columns of $Z=\bra{\ze^{\sig(i)\sig(j)}}_{0\leq i,j\leq N'-1}$. All $Z_k$ are principal minors of $\mc{F}_{N'}$ with permuted rows and columns, and each $V_{\sig(k),\sig(k)}$ is a Vandermonde matrix with determinant given by 
\[\prod_{1\leq i<j\leq M_{\sig(k)}}(m_{\sig(k),j}-m_{\sig(k),i}).\]
Thus, we obtain from \eqref{eq:Nevaluation1} that 
\begin{equation}\label{eq:Nfinalevaluation}
\begin{split}
&P(\underbrace{1,\dotsc,1}_{M_0 \text{ terms}},\underbrace{\ze,\dotsc,\ze}_{M_1 \text{ terms}}, \dotsc, \underbrace{\ze^{N'-1},\dotsc,\ze^{N'-1}}_{M_{N'-1} \text{ terms}})\\
&=\frac{\prod_{k=0}^{N'-1}(\det Z_k)^{M_{\sig(k)}-M_{\sig(k+1)}}\cdot\prod_{k=0}^{N'-1}\prod_{1\leq i<j\leq M_{\sig(k)}}(m_{\sig(k),j}-m_{\sig(k),i})}{\prod_{a=1}^{N'-1}\ze^{aM_a(M_a-1)/2} \cdot \prod_{0\leq a<b\leq N'-1}(\ze^b-\ze^a)^{M_aM_b} \cdot \prod_{a=0}^{N'-1}\prod_{k=0}^{M_a-1}k!}.
\end{split}
\end{equation}
Taking the norm $\Norm=\Norm^{\QQ(\ze)}_\QQ$ on both sides of \eqref{eq:Nfinalevaluation}, we get an integer divisible by $p$ on the left-hand side, whereas on the right-hand side we obtain an integer that is not divisible by $p$; indeed, $p\nmid\Norm{\det(Z_k)}$, $0\leq k\leq N'-1$ by hypothesis, $\Norm(\ze)=\pm1$ by \eqref{eq:normrootofunity}, $\Norm(\ze^a-\ze^b)\mid N'$ by \eqref{eq:valuecyclotomic0}, each $m_{\sig(k),j}-m_{\sig(k),i}$ is not divisible by $p$, and the product of factorials is also not divisible by $p$. This obviously establishes a contradiction, due to our assumption that $D_{\mc{M},\mc{L}}=0$. Therefore, the $N'$-principal minor $D_{\mc{M},\mc{L}}$ is indeed non-zero, as desired.
\end{proof}

Theorem \ref{thm:inductionstep} has a series of consequences.

\begin{cor}\label{cor:pqCheb}
    Let $N=pq$, a product of two distinct primes. Suppose that either $\mc{F}_p$ has the $q$-Chebotarev property or $\mc{F}_q$ has the $p$-Chebotarev property. Then, all principal minors of $\mc{F}_N$ are non-zero. 
\end{cor}

In particular, we get the following:

\begin{cor}\label{cor:smallmultp}
    Let $N=pq$, where $p, q$ are distinct primes and $p\in\set{2,3,5,7}$. Then all principal minors of $\mc{F}_N$ are non-zero.
\end{cor}

\begin{proof}
    It follows immediately from Corollary \ref{cor:pqCheb} and Propositions \ref{prop:235cheb} and \ref{prop:7cheb}.
\end{proof}

The next corollary generalizes the result of Loukaki \cite{Lo24}:

\begin{cor}
    Let $N=pq$, a product of two distinct primes. If $r=\ord_p(q)$ and $q>\Ga^{\frac{p-1}{r}}$, where $\Ga$ as in \eqref{eq:Gammadefinition}, then all principal minors of $\mc{F}_N$ are non-zero.
\end{cor}

\begin{proof}
    It follows immediately from Theorems \ref{thm:genZhang} and \ref{thm:inductionstep}.
\end{proof}

At this point, it is natural to pose the following question, which might be of independent interest in number theory:

\smallskip
\noindent
\emph{For which pairs of distinct primes $(p, q)$ do we have that $\mc{F}_p$ has the $q$-Chebotarev property or that $\mc{F}_q$ has the $p$-Chebotarev property?}
\smallskip

In view of Corollary \ref{cor:pqCheb}, finding such a pair $(p,q)$ establishes Conjecture \ref{conj:squarefree} for the case $N = pq$. 
However, such pairs do not seem to be generic. For instance, consider a pair of primes $(p,q)$ such that $q\equiv1\bmod 4p$ and $q\mid\Norm(\om^{b_1}+\dotsb+\om^{b_n})$, for some $B=\set{b_1,\dotsc,b_n}\subset\ZZ_p$. The fact that $q\equiv1\bmod 4p$ implies that $\mc{F}_q$ does not satisfy the $p$-Chebotar\"ev property, by Theorem \ref{thm:Gauss}. On the other hand, the fact that $q\mid\Norm(\om^{b_1}+\dotsb+\om^{b_n})$ implies that $\mc{F}_p$ does not satisfy the $q$-Chebotar\"ev property, by equation \eqref{eq:resultant}. A simple calculation for small primes on sets $B$ with only $3$ or $4$ elements, yields the pairs $(p,q) = (11,89), (13,53), (17,953), (19,457)$, as shown by the equations below:
\begin{align*}
    \Norm(1+\om_{11}+\om_{11}^2+\om_{11}^5) &=89\\
    \Norm(1+\om_{13}+\om_{13}^3) &=53\\
    \Norm(1+\om_{17}+\om_{17}^2+\om_{17}^5) &=953\\
    \Norm(1+\om_{19}+\om_{19}^3) &=457,
\end{align*}
where $\om_p=e^{-2\pi i/p}$. We suspect that for every prime $p\geq11$, there is a prime $q\neq p$, such that neither $\mc{F}_p$ satisfies the $q$-Chebotar\"ev property, nor does $\mc{F}_q$ satisfy the $p$-Chebotar\"ev property. Thus, in order to prove Conjecture \ref{conj:squarefree}, even for the special case $N=pq$, the method used in the proof of Theorem \ref{thm:inductionstep} needs to be refined; we should be able to address the local obstructions that appear when $p$ divides some of the norms $\Norm\det (Z_k)$ in \eqref{eq:Nfinalevaluation}, by estimating the $p$-adic valuation on both sides of \eqref{eq:Nfinalevaluation}, proving the desired result by contradiction, perhaps by showing that the $p$-adic valuations on the two sides of this equation are different.

We continue by providing a positive result in the special case where $N = 2 \cdot 3 \cdot p$ with $p>3$ prime.

\begin{prop}\label{prop:N-6p}
    Let $N=6p$, where $p>3$ is a prime. Then, all principal minors of $\mc{F}_N$ are non-zero.
\end{prop}

\begin{proof}
    By Theorem \ref{thm:inductionstep}, it suffices to check that all principal minors of $\mc{F}_6$ are non-zero in any characteristic $p>3$. By the remark following Proposition \ref{prop:unitarycomplementarity}, it suffices to check the principal minors up to size $3\times 3$, that is, it suffices to verify $\ol{D}_{A,A}\neq0$ for every $A\subset\ZZ_6$ with $\abs{A} \leq 3$. 

    If $\abs{A} = 2$, then $A$ is trivially an arithmetic progression (say, with step size $d$), and since $d^2\not\equiv0\bmod 6$, we have from Proposition~\ref{prop:arithmeticprogression} that $\ol{D}_{A,A}\neq0$. 
    
    Next, let $\abs{A}=3$. If $3\in A-A$, then by Proposition \ref{prop:affinenorminvariance} we may assume $\set{0,3}\subset A$. Then, again by Proposition \ref{prop:affinenorminvariance}, we may assume $A=\set{0,1,3}$, as any $3$-element subset of $\ZZ_6$ can be mapped via an affine transformation to $\set{0,1,3}$. Then
    \[D_{A,A}=\begin{vmatrix}
        1 & 1 & 1\\
        1 & \om & -1\\
        1 & -1 & -1
    \end{vmatrix}=-2\om-2=-2(1+\om)  \]
    with norm given by $\Norm (D_{A,A}) = 4\Norm(1+\om)=12$, which is not divisible by any prime $p>3$. 
    
    Now assume that $3\notin A-A$. Then, all three elements of $A$ must be distinct $\bmod3$. Using Proposition \ref{prop:affinenorminvariance}, we can simultaneously translate the elements of $A$ so that $0\in A$. 
    If $A$ contains a primitive element, then without loss of generality, $1\in A$, by Proposition \ref{prop:affinenorminvariance}, so that $A=\set{0,1,2}$ or $A=\set{0,1,5}$. Both are arithmetic progressions of step $1$, hence $\ol{D}_{A,A}\neq0$ in characteristic $p$, by Proposition \ref{prop:arithmeticprogression}. If $A$ does not contain a primitive element, then we must have $A=\set{0,2,4}$, which is an arithmetic progression with step $2$, and $2(a-a')\not\equiv0\bmod6$ for any $a\neq a'\in\ZZ_6$. Again, by Proposition \ref{prop:arithmeticprogression} we get $\ol{D}_{A,A}\neq0$, completing the proof.
\end{proof}

The above Proposition is the first instance where Theorem \ref{thm:inductionstep} was used as an \emph{induction step}, 
towards the goal of settling Conjecture \ref{conj:squarefree} for square-free numbers that are product of more than two primes. 
A more general result is as follows:

\begin{prop}\label{prop:inductivetool}
    Let $N$ be square-free, and let $A,B\ssq\ZZ_N$, such that $m=\abs{A}=\abs{B}$ and $D_{A,B}\neq0$. Then, for any prime $p\nmid N$ with $p>m^{m\vphi(N)/2}$, we also have $p\nmid \Norm(D_{A,B})$, i.e. $\ol{D}_{A,B}\neq0$.
\end{prop}

\begin{proof}
    Each column of $D_{A,B}$ has length $\sqrt{m}$. Therefore, the modulus of $D_{A,B}$ is at most the product of the lengths of the columns, i.e.
    \[\abs{D_{A,B}}\leq m^{m/2}.\]
    The standard complex conjugation in $\CC$, when restricted on $\QQ(\om)$, may be identified with an element $\tau\in\Gal(\QQ(\om)/\QQ)$, namely, the one induced by $\om\mapsto \om^{-1}$. Denote $\al=D_{A,B}$; this shows that
    \[0<\al\cdot\tau(\al)\leq m^m.\]
    Taking norms, we get $0<(\Norm(\al))^2\leq m^{m\vphi(N)}$, yielding $0<\Norm(\al)\leq m^{m\vphi(N)/2}$. Thus, if $p>m^{m\vphi(N)/2}$, we get the desired result.
\end{proof}

We are now ready to prove Theorem \ref{thm:mainresult2}.

\begin{proof}[Proof of Theorem \ref{thm:mainresult2}]
    We will show inductively that all principal minors of $\mc{F}_{P_j}$ are non-zero for $1 \leq j \leq k$. 
    For $j=1$, all principal minors of $\mc{F}_{P_1} = \mc{F}_{p_1}$ are non-zero by Chebotar\"ev's theorem. 
    Now, suppose that all principal minors of $\mc{F}_{P_j}$ are non-zero for $1 \leq j \leq n$, where $n < k$.
    We will show that these principal minors are also non-zero in characteristic $p_{n+1}$, where 
    \[p_{n+1}>\bra{\frac{P_n}{2}}^{P_n\vphi(P_n)/4}.\]
    For this purpose, let $D_{A,A}$ be any minor of $\mc{F}_{P_n}$. Without loss of generality, we may assume $\abs{A} = m \leq\frac{P_n}{2}$ by the remark following Proposition \ref{prop:unitarycomplementarity}. 
    Since 
    \[p_{n+1}>\bra{\frac{P_n}{2}}^{P_n\vphi(P_n)/4} \geq m^{m\vphi(P_n)/2},\]    
    we get $p_{n+1}\nmid \Norm(D_{A,A})$ by Proposition \ref{prop:inductivetool}, so $D_{A,A}$ is also non-zero in characteristic $p_{n+1}$. Then all principal minors of $\mc{F}_{P_{n+1}}$ are non-zero by Theorem \ref{thm:inductionstep}. 
    By induction, we conclude that all principal minors of $\mc{F}_{P_j}$ are non-zero for $1 \leq j \leq k$. 
\end{proof}

We conclude this Section by presenting certain classes of non-zero principal minors $D_{A,A}$ of $\mc{F}_N$ for \emph{all} square-free $N$. Retracing the steps in the proof of Theorem \ref{thm:inductionstep}, let $D_{\mc M, \mc L}$ be an $N'$-principal minor that admits a permutation $\sig$ of $\ZZ_{N'}$ such that 
\begin{equation}\label{eq:arithmeticprincipal}
M_{\sig(0)}\geq M_{\sig(1)}\geq\dotsb\geq M_{\sig(N'-1)},    
\end{equation}
where $\sig(0),\dotsc,\sig(N'-1)$ forms an arithmetic progression in $\ZZ_{N'}$, say $\sig(j)=aj+b$ with $\gcd(a,N')=1$. Continuing with the calculation starting from \eqref{eq:Ndeterminant2} and leading to \eqref{eq:Nfinalevaluation}, we note that the matrix $Z_k=\bra{\ze^{\sig(i)\sig(j)}}_{0\leq i,j\leq k-1}$ is a standard Vandermonde matrix up to multiplication by a root of unity, with determinant given by 
\[
\begin{split}
\det Z_k
&=\ze^{\sig(0)(\sig(0)+\dotsb+\sig(k-1))} \, V_{\bs s}(\ze^{a\sig(0)},\dotsc,\ze^{a\sig(k-1)}) \\
&=\ze^{\sig(0)(\sig(0)+\dotsb+\sig(k-1))}\prod_{0\leq i<j\leq k-1}(\ze^{a\sig(j)}-\ze^{a\sig(i)}).    
\end{split}
\]
Substituting this value into the right-hand side of \eqref{eq:Nfinalevaluation}, we find that its norm is not divisible by $p$, yielding a similar contradiction and thereby establishing that $D_{\mc M,\mc L}\neq0$. 

\begin{defn}
    Let $N$ be a square-free number, and let $p$ be a prime dividing $N$, i.e., $N=pN'$ with $p \nmid N'$. 
    A minor $D_{\mc M,\mc L}$ of $\mc F_N$ is called \emph{arithmetic} $N'$-principal, if $\mc M$ satisfies \eqref{eq:arithmeticprincipal}.
\end{defn}

For this particular class of $N'$-principal minors, the condition that all principal minors of $\mc F_{N'}$ be non-zero in characteristic $p$, appearing in the statement of Theorem \ref{thm:inductionstep}, is not needed.
Consequently, we obtain the following:

\begin{thm}
    Let $N$ be a square-free number, and let $p$ be a prime dividing $N$, i.e., $N=pN'$ with $p \nmid N'$. 
    Then, any arithmetic $N'$-principal minor is non-zero.
\end{thm}
\bigskip\bigskip

\bigskip
\bigskip

\bibliographystyle{amsplain}
\bibliography{bibliography}

@Book{FR13,
	author    = {Simon Foucart and Holger Rauhut},
  title     = {A Mathematical Introduction to Compressive Sensing},
  publisher = {Birkh{\"a}user},
	address   = "Basel",
  year      = {2013},
}

@book{HJ13,
  author = "R.A. Horn and C.R. Johnson",
  title = "Matrix Analysis, second ed.",
  address = "New York",
  publisher = {Cambridge University Press},
  year = "2013"
}

@misc{CLMP25-2,
      title={Chebotar\"ev type results for composite group order, {P}art 2: natural numbers containing squares}, 
      author={Caragea, A. and Lange, F. and Lee, D. G. and Malikiosis, R. and Pfander, G. E.},
      year={2025, in prepration},
}

@misc{Lo24,
      title={{C}hebotarev's theorem for roots of unity of square free order}, 
      author={Loukaki, M.},
      year={2024},
      eprint={2412.08600},
      archivePrefix={arXiv},
      primaryClass={math.NT},
      url={https://arxiv.org/abs/2412.08600}, 
      note={\url{https://arxiv.org/abs/2412.08600}},
}

@misc{Fr03,
      title={Simple proof of {C}hebotarev's theorem on roots of unity}, 
      author={P. E. Frenkel},
      year={2003},
      eprint={math/0312398},
      archivePrefix={arXiv},
      primaryClass={math.AC},
      url={https://arxiv.org/abs/math/0312398}, 
      note={\url{https://arxiv.org/abs/math/0312398}},
}

@misc{CL24,
      title={On the principal minors of {F}ourier matrices}, 
      author={Caragea, A. and Lee, D. G.},
      year={2024},
      eprint={2409.09793},
      archivePrefix={arXiv},
      primaryClass={math.FA},
      url={https://arxiv.org/abs/2409.09793}, 
      note={\url{https://arxiv.org/abs/2409.09793}},
}

@article{CL22,
    author={Caragea, A. and Lee, D. G.},
    title={A note on exponential {R}iesz bases},
    journal={Sampling Theory, Signal Processing, and Data Analysis},
    year={2022},
    month={Jul},
    day={28},
    volume={20},
    number={2},
    pages={13}, 
    issn={2730-5724},
    doi={10.1007/s43670-022-00031-9},
    url={https://doi.org/10.1007/s43670-022-00031-9}
}

@article {Chebotarev,
    AUTHOR = {Tschebotareff, N.},
     TITLE = {Die {B}estimmung der {D}ichtigkeit einer {M}enge von
              {P}rimzahlen, welche zu einer gegebenen {S}ubstitutionsklasse
              geh\"oren},
   JOURNAL = {Math. Ann.},
  FJOURNAL = {Mathematische Annalen},
    VOLUME = {95},
      YEAR = {1926},
    NUMBER = {1},
     PAGES = {191--228},
      ISSN = {0025-5831,1432-1807},
   MRCLASS = {99-04},
       DOI = {10.1007/BF01206606},
       URL = {https://doi.org/10.1007/BF01206606},
}

@article {DVB08,
    AUTHOR = {Delvaux, S. and Van Barel, M.},
     TITLE = {Rank-deficient submatrices of {F}ourier matrices},
   JOURNAL = {Linear Algebra Appl.},
  FJOURNAL = {Linear Algebra and its Applications},
    VOLUME = {429},
      YEAR = {2008},
    NUMBER = {7},
     PAGES = {1587--1605},
      ISSN = {0024-3795,1873-1856},
   MRCLASS = {15A23},
  MRNUMBER = {2444345},
MRREVIEWER = {Abbas\ Salemi},
       DOI = {10.1016/j.laa.2008.04.043},
       URL = {https://doi.org/10.1016/j.laa.2008.04.043},
}

@article {GenVandermonde,
    AUTHOR = {Evans, R. J. and Isaacs, I. M.},
     TITLE = {Generalized {V}andermonde determinants and roots of unity of
              prime order},
   JOURNAL = {Proc. Amer. Math. Soc.},
  FJOURNAL = {Proceedings of the American Mathematical Society},
    VOLUME = {58},
      YEAR = {1976},
     PAGES = {51--54},
      ISSN = {0002-9939,1088-6826},
   MRCLASS = {15A15},
MRREVIEWER = {W.\ Ledermann},
       DOI = {10.2307/2041358},
       URL = {https://doi.org/10.2307/2041358},
}

@article {KurshanOdlyzko,
    AUTHOR = {Kurshan, R. P. and Odlyzko, A. M.},
     TITLE = {Values of cyclotomic polynomials at roots of unity},
   JOURNAL = {Math. Scand.},
  FJOURNAL = {Mathematica Scandinavica},
    VOLUME = {49},
      YEAR = {1981},
    NUMBER = {1},
     PAGES = {15--35},
      ISSN = {0025-5521,1903-1807},
   MRCLASS = {12A35 (12E10)},
MRREVIEWER = {Lawrence\ Washington},
       DOI = {10.7146/math.scand.a-11919},
       URL = {https://doi.org/10.7146/math.scand.a-11919},
}

@book {Marcus,
    AUTHOR = {Marcus, D. A.},
     TITLE = {Number fields},
    SERIES = {Universitext},
   EDITION = {Second},
      NOTE = {With a foreword by Barry Mazur},
 PUBLISHER = {Springer, Cham},
      YEAR = {2018},
     PAGES = {xviii+203},
      ISBN = {978-3-319-90232-6; 978-3-319-90233-3},
   MRCLASS = {11-01 (11Rxx 11Txx 12-01)},
       DOI = {10.1007/978-3-319-90233-3},
       URL = {https://doi.org/10.1007/978-3-319-90233-3},
}

@article {TaoUncertainty05,
    AUTHOR = {Tao, T.},
     TITLE = {An uncertainty principle for cyclic groups of prime order},
   JOURNAL = {Math. Res. Lett.},
  FJOURNAL = {Mathematical Research Letters},
    VOLUME = {12},
      YEAR = {2005},
    NUMBER = {1},
     PAGES = {121--127},
      ISSN = {1073-2780},
   MRCLASS = {11B75},
MRREVIEWER = {Ben\ Joseph\ Green},
       DOI = {10.4310/MRL.2005.v12.n1.a11},
       URL = {https://doi.org/10.4310/MRL.2005.v12.n1.a11},
}

@book {washington,
    AUTHOR = {Washington, L. C.},
     TITLE = {Introduction to cyclotomic fields},
    SERIES = {Graduate Texts in Mathematics},
    VOLUME = {83},
   EDITION = {Second},
 PUBLISHER = {Springer-Verlag, New York},
      YEAR = {1997},
     PAGES = {xiv+487},
      ISBN = {0-387-94762-0},
   MRCLASS = {11R18 (11-01 11-02 11R23)},
MRREVIEWER = {T.\ Mets\"ankyl\"a},
       DOI = {10.1007/978-1-4612-1934-7},
       URL = {https://doi.org/10.1007/978-1-4612-1934-7},
}

@article {pChebotarev,
    AUTHOR = {Zhang, G.},
     TITLE = {On the {C}hebotar\"ev theorem over finite fields},
   JOURNAL = {Finite Fields Appl.},
  FJOURNAL = {Finite Fields and their Applications},
    VOLUME = {56},
      YEAR = {2019},
     PAGES = {97--108},
      ISSN = {1071-5797,1090-2465},
   MRCLASS = {15B33 (94B05)},
MRREVIEWER = {Martin\ Mereb},
       DOI = {10.1016/j.ffa.2018.11.004},
       URL = {https://doi.org/10.1016/j.ffa.2018.11.004},
}

@misc{EK25,
      title={Real and finite field versions of {C}hebotar\"ev's theorem}, 
      author={Emmrich, T. and Kunis, S.},
      year={2025},
      eprint={2506.02947},
      archivePrefix={arXiv},
      primaryClass={math.NA},
      url={https://arxiv.org/abs/2506.02947}, 
      note={\url{https://arxiv.org/abs/2506.02947}},
}

@article {CMN24,
    AUTHOR = {Cabrelli, C. and Molter, U. and Negreira, F.},
     TITLE = {Weaving {R}iesz bases},
   JOURNAL = {J. Fourier Anal. Appl.},
  FJOURNAL = {The Journal of Fourier Analysis and Applications},
    VOLUME = {31},
      YEAR = {2025},
    NUMBER = {1},
     PAGES = {Paper No. 4, 20},
      ISSN = {1069-5869,1531-5851},
   MRCLASS = {42C15 (47A15 94A20)},
  MRNUMBER = {4837946},
       DOI = {10.1007/s00041-024-10134-7},
       URL = {https://doi.org/10.1007/s00041-024-10134-7},
}

\end{document}